\documentclass[11pt]{article}

\usepackage{amsfonts,amsmath,amsthm,latexsym,color,epsfig,mathrsfs}
\newtheorem{thm}{Theorem}[section]

\newcommand{\A}{\mathcal{A}}
\newcommand{\B}{\mathcal{B}}
\newcommand{\C}{\mathcal{C}}

\title{Forbidding a Set Difference of Size $1$}
\author{
Imre Leader\thanks{Department of Pure Mathematics and Mathematical Statistics, Centre for Mathematical Sciences, 
Cambridge CB3 0WB, United Kingdom. E-mail: I.Leader@dpmms.cam.ac.uk} \and 
Eoin Long\thanks{School of Mathematical Sciences, Queen Mary University of London, Mile End Road, London E1 4NS, United
Kingdom. E-mail: E.P.Long@qmul.ac.uk}
}

\begin{document}

\maketitle

\begin{abstract}
 How large can a family ${\mathcal A} \subset {\mathcal P}[n]$ be if it does not contain $A,B$ with $|A\setminus B| = 1$? Our aim in this paper is to show that any such family has size at most $\frac{2+o(1)}{n}\binom {n}{\lfloor n/2\rfloor }$. This is tight up to a multiplicative constant of $2$. We also obtain similar results for families ${\mathcal A} \subset {\mathcal P}[n]$ with $|A\setminus B| \neq k$, showing that they satisfy $|{\mathcal A}|  \leq \frac{C_k}{n^k}\binom {n}{\lfloor n/2\rfloor }$, where $C_k$ is a constant depending only on $k$.
\end{abstract}

\section{Introduction}

A family ${\mathcal A} 
\subset {\mathcal P}[n] = {\mathcal P} ( \{ 1,\ldots,n \})$ is 
said to be a \emph{Sperner family} or \emph{antichain} if $A\not \subset B$ 
for all distinct $A,B\in {\mathcal A}$. Sperner's theorem \cite{sper}, one of the earliest 
result in extremal combinatorics, states that every Sperner family ${\mathcal A}\subset {\mathcal P}[n]$ satisfies 
\begin{equation}
 \label{SperInequality}
 |\A| \leq \binom{n}{\lfloor n/2\rfloor }. 
\end{equation}
[We remark that this paper is self-contained; for background on Sperner's
theorem and related 
results see \cite{com}.]

Kalai \cite{kal} noted that the Sperner condition can be rephrased as follows: ${\mathcal A}$ does not contain two sets $A$ and 
$B$ such that, in the unique subcube of ${\mathcal P}[n]$ spanned by $A$ and $B$, $A$ is the bottom point and $B$ is the top 
point. He asked: what happens if we forbid $A$ and $B$ to be at a different position in this subcube? In particular, he 
asked how large ${\mathcal A} \subset {\mathcal P}[n]$ can be if we forbid $A$ and $B$ to be at a `fixed ratio' $p:q$ in this 
subcube. That is, we forbid $A$ to be $p/(p+q)$ of the way up this subcube and $B$ to be $q/(p+q)$ of the way up this subcube. 
Equivalently, $q|A\setminus B| \neq p|B\setminus A|$ for all 
distinct $A,B\in {\mathcal A}$. Note that the Sperner condition corresponds 
to taking $p=0$ and $q=1$. In \cite{leadlong}, we gave an asymptotically tight answer for all ratios $p:q$, showing that one cannot improve on the `obvious'
example, namely the $q-p$ middle layers of ${\mathcal P}[n]$. 
\begin{thm}[\cite{leadlong}]
 \label{tiltedsperner}
  Let $p,q$ be coprime natural numbers with $q\geq p$. Suppose ${\mathcal A} \subset {\mathcal P}[n]$ does not contain distinct $A,B$ with
  $q|A\setminus B| = p|B\setminus A|$. Then 
\begin{equation}
   |\A| \leq (q-p + o(1))\binom {n}{\lfloor n/2\rfloor }.
\end{equation}
\end{thm}
Up to the $o(1)$ term, this is best possible. Indeed, the proof of Theorem \ref{tiltedsperner} in \cite{leadlong} also gives the exact maximum size of such ${\mathcal A}$ for 
infinitely many values of $n$.

Another natural question considered in \cite{leadlong} asks how large a family ${\mathcal A} \subset {\mathcal P}[n]$ can be if, 
instead of forbidding a fixed ratio, we forbid a `fixed distance' in these subcubes. For example, how large can ${\mathcal A} \subset {\mathcal P}[n]$ be if 
$A$ is not at distance $1$ from the bottom of the subcube spanned with $B$ for all $A,B \in {\mathcal A}$? Equivalently, 
$|A\setminus B| \neq 1$ for all $A,B \in {\mathcal A}$. Here the following family $\mathcal A^*$ provides a lower bound: 
let $\mathcal A^*$ consist of all sets $A$ of size $\lfloor n/2\rfloor $ such that $\sum _{i\in A} i \equiv r \pmod {n}$ 
where $r\in \{0,\ldots ,n-1\}$ is chosen to maximise $|\mathcal A^*|$. Such a choice of $r$ gives 
$| \mathcal A^* |\geq {\frac {1}{n}} {\binom{n} {\lfloor n/2\rfloor }}$. 
Note that if we had $|A\backslash B|=1$ for some 
$A,B\in \mathcal A^*$, since $|A|=|B|$, we would also 
have $|B\backslash A|=1$ -- letting $A\backslash B=\{i\}$ and 
$B\backslash A=\{j\}$ we then have $i-j\equiv 0 \pmod {n}$ giving $i=j$, a contradiction. 

In \cite{leadlong}, we showed that any such family ${\mathcal A} \subset {\mathcal P}[n]$ satisfies $|\A | \leq \frac{C}{n}2^n = 
O(\frac{1}{n^{1/2}}\binom {n}{\lfloor n/2\rfloor })$ for some absolute constant $C>0$. We conjectured that the family ${\mathcal A}^{*}$ 
constructed in the previous paragraph is asymptotically maximal (Conjecture 5 of \cite{leadlong}). In Section 2, we prove that this is true up to a factor of $2$.

\begin{thm}
 \label{stickoutexactly1}
 Suppose that ${\mathcal A} \subset {\mathcal P}[n]$ is a family of sets with $|A\setminus B| \neq 1$ for all $A,B\in {\mathcal A}$. 
 Then $|{\mathcal A}| \leq \frac{2 + o(1)}{n} \binom {n}{\lfloor n/2\rfloor }$.
\end{thm}

One could also ask what happens if we forbid a fixed set difference of size $k$, instead of $1$ (where we think of $k$ as fixed and $n$ as varying). This turns out to be harder. In \cite{leadlong} we noted that the following family $\A ^{*}_k\subset {\mathcal P}[n]$ gives a lower bound of $\frac{1}{n^k}\binom {n}{\lfloor n/2\rfloor }$: supposing n is prime, let ${\mathcal A}_k^*$ consist of all sets $A$ of size $\lfloor n/2\rfloor $ which satisfy $\sum _{i\in A} i^d \equiv 0\pmod n$ for all $1\leq d\leq k$. In Section 3 we prove that this is also best possible up to a multiplicative constant.

\begin{thm} 
 \label{stickout k}
Let $k\in {\mathbb N}$. Suppose that ${\mathcal A} \subset {\mathcal P}[n]$ with $|A \setminus B| \neq k$ for all $A,B \in {\mathcal P}[n]$. Then $|{\mathcal A}| \leq \frac{C_k}{n^k}\binom {n}{\lfloor n/2\rfloor }$, where $C_k$ is a constant depending only on $k$.
\end{thm}

Our notation is standard. We write $[n]$ for  $\{ 1,\ldots,n \}$, and $[a,b]$ for the interval  $\{ a,\ldots,b \}$. For a set $X$, we write ${\mathcal P}(X)$ for the power set of $X$ and $X^{(k)}$ for collection of all $k$-sets in $X$. We often suppress integer-part signs.

\section{Proof of Theorem \ref{stickoutexactly1}}

Our proof of Theorem \ref{stickoutexactly1} uses Katona's averaging method (see \cite{katona}), but modified in a key way. Ideally here, as in 
the proof of Sperner's theorem or Theorem \ref{tiltedsperner}, 
we would find configurations of sets covering ${\mathcal P}[n]$, so that each configuration has at most $C/n^{3/2}$ proportion of its 
elements in ${\mathcal A}$, for any family ${\mathcal A}$ satisfying $|A\setminus B|\neq 1$ for $A,B\in {\mathcal A}$. Then, provided 
these configurations cover ${\mathcal P}[n]$ uniformly,
we could count incidences between elements of ${\mathcal A}$ and these configurations 
to get an upper bound on $|{\mathcal A}|$.

However, we do not see how to find such configurations. So instead our approach
is as follows. Rather than insisting that each of the sets in our configuration
is well-behaved (in the sense above), we will only require that  
\emph{most} of them have at most $C/n^{3/2}$ proportion of their elements in ${\mathcal A}$. It turns out that this can be achieved, and that it is 
good enough for our purposes.

\begin{proof}
 To begin with, remove all elements in ${\mathcal A}$ of size smaller than $n/2 - n^{2/3}$ or larger than 
 $n/2 + n^{2/3}$. By Chernoff's inequality (see Appendix A of \cite{alonspen}), we have removed at most $o(\frac{1}{n}\binom {n}{n/2})$ sets. Let ${\mathcal B}$ denote the remaining sets in $\mathcal A$. It suffices to show that $|{\mathcal B}| \leq \frac {2+o(1)}{n}\binom {n}{n/2}$.

We write $I = [1, n/2 + n^{2/3}]$ and $J = [n/2 + n^{2/3} + 1,n]$ so 
that $[n] = I \cup J$. Let us choose a permutation $\sigma \in S_n$ uniformly at random. Given this choice of $\sigma $, for all $i\in I$, $j\in J$ let $C_{i,j} = \{\sigma (1),\ldots \sigma (i)\} 
 \cup \{\sigma (j)\}$. Let ${\mathcal C}_j = \{C_{i,j}: i\in I\}$, and call these sets `partial chains'.  Also let 
${\mathcal C} = \bigcup _{j\in J} {\mathcal C}_j$.

Now, for any choice of $\sigma \in S_n$, at most one of the partial chains of ${\mathcal C}$ can contain an element of ${\mathcal B}$. Indeed, suppose both $C_{i_1,j_1} = C_{i_1} \cup \{\sigma (j_1)\}$ and 
 $C_{i_2,j_2} = C_{i_2} \cup \{\sigma (j_2)\}$ lie in 
 ${\mathcal A}$ for distinct $j_1, j_2 \in J$. Since $C_{i_1}$ and $C_{i_2}$ are elements of a 
 chain, without loss of generality we may assume $C_{i_1} \subset C_{i_2}$. But then $(C_{i_1} \cup \{\sigma (j_1)\}) 
 \setminus (C_{i_2} \cup \{\sigma (j_2)\}) = \{ \sigma (j_1)	\}$, which contradicts $|A\setminus B| \neq 1$ for all $A,B \in {\mathcal B}$.

Note that the above bound alone does not guarantee the upper bound on $|{\mathcal A}|$ stated in the theorem, since a fixed partial chain ${\mathcal C}_i$ may contain many elements of ${\mathcal A}$. We now show that this cannot happen too often.

 For $i\in I$ and $j\in J$, let $X_{i,j}$ denote the random variable given by
  \[X_{i,j} = \bigg\{ \begin{array}{ll}
         1 & \mbox{ if } C_{i,j} \in {\mathcal B} \mbox{ and } C_{k,j} \notin {\mathcal B} 
               \mbox{ for } k < i;\\
         0 & \mbox{ otherwise.}\end{array} \]
 From the previous paragraph, we have 
  \begin{equation}
   \label{Xijinequality}
   \sum _{i,j} X_{i,j} \leq 1
  \end{equation}
 where both here and below the sum is taken over all $i \in I$ and $j\in J$. Taking expectations on both sides of (\ref{Xijinequality}) this gives 
  \begin{equation}
   \label{expectationupperbound}
   \sum _{i, j} \mathbb{E}(X_{i,j}) \leq 1.
  \end{equation}
 Rearranging we have
  \begin{equation}
   \label{expectationXij}
   \begin{split}
   \sum _{i,j} \mathbb{E}(X_{i,j}) &= \sum _{i,j} \sum _{B \in {\mathcal B}} \mathbb {P}(C_{i,j} = B \mbox{ and } C_{k,j} \notin {\mathcal B} 
   \mbox{ for } k <i).
   \end{split} 
 \end{equation}

 We now bound $\mathbb {P}(C_{i,j} = B \mbox{ and } C_{k,j}\notin {\mathcal B} \mbox{ for } k<i)$ for sets $B\in {\mathcal B}$. 
 Note that we can only have $C_{i,j} = B$ if $|B| = i+1$. Furthermore, for such $B$, since $C_{i,j}$ is equally likely to be any subset of $[n]$ of size $i+1$, we have $\mathbb{P}(C_{i,j} = B) = 1 / \binom {n}{i+1}$. We will show that for all such $B$
 \begin{equation}
  \label{looselittlebyrestricting}
  \mathbb {P}( C_{i,j} = B \mbox { and } C_{k,j} \notin {\mathcal B} \mbox{ for } k<i) 
  =  (1 - o(1)) \mathbb {P}( C_{i,j} = B )
 \end{equation}
 To see this, note that given any set $D \subset [n]$, there is at most one element $d\in D$ such that $D-d \in \mathcal B$. 
 Indeed, $|(D-d')\setminus (D-d)| =1$ for any distinct choices of $d,d'\in D$. Recalling that 
 $C_{k,j} = C_{i,j} -\{\sigma (k+1),\ldots ,\sigma (i)\}$ for all $k<i$ and that $\sigma (k+1)$ is chosen uniformly at random from 
 the $k+1$ elements of $C_{k+1,j}-\{\sigma (j)\}$, we see that for $k + 1\geq n/2 - n^{2/3}$ we have 
  \begin{equation}
   \label{Ckjprobability}
   {\mathbb P}(C_{k,j} \notin {\mathcal B} | C_{k+1,j},\ldots , C_{i,j}) \geq (1 - \frac {1}{k+1}) \geq (1 - \frac {1}{n/2 - n^{2/3}}).
  \end{equation}
 Also, since $\mathcal B$ contains no sets of size less than $n/2 - n^{2/3}$, for $k + 1 < n/2 - n^{2/3}$ we have
  \begin{equation}
   \label{trivialboundCkj}
    {\mathbb P}(C_{k,j} \notin {\mathcal B} | C_{k+1,j},\ldots , C_{i,j}) = 1.
  \end{equation}
 But now by repeatedly applying (\ref{Ckjprobability}) and (\ref{trivialboundCkj}) we get that 
 for any $B$ of size $i+1\in [n/2 - n^{2/3}, n/2 + n^{2/3}]$ we have 
  \begin{equation*}
   \begin{split}
    {\mathbb P} ( C_{i,j} = B \mbox{ and } C_{k,j} \notin { \mathcal B } \mbox{ for } k < i ) 
     & \geq (1- \frac {1} {n/2 -n^{2/3}} )^{(i - n/2 -n^{2/3})} {\mathbb P}( C_{i,j} = B )\\
     & \geq (1- \frac{1}{n/2 -n^{2/3}})^{2n^{2/3}} {\mathbb P}(C_{i,j} = B)\\
     & = (1 - o(1)){\mathbb P} ( C_{i,j} = B ).
   \end{split}
  \end{equation*}
 Now combining (\ref{looselittlebyrestricting}) with (\ref{expectationupperbound}) and (\ref{expectationXij}) we obtain 
 \begin{equation*}
  \begin{split}
  1& \geq \sum _{i, j} \mathbb{E}(X_{i,j}) \\
 & = \sum _{i,j} \sum _{B \in {\mathcal B}} \mathbb {P}(C_{i,j} = B \mbox{ and } C_{k,j} \notin {\mathcal B} 
   \mbox{ for } k <i)\\
 & = \sum _{i,j} 
     \sum _{B \in {\mathcal B}^{(i+1)}} (1 - o(1)) \mathbb {P}( C_{i,j} = B)\\
 & = (1 - o(1)) \sum _{i,j} \frac { |{\mathcal B}^{(i+1)}| }{\binom {n}{i+1}}\\
 & = (1-o(1))|J| \sum _i \frac { |{\mathcal B}^{(i+1)}| }{\binom {n}{i+1}}.
  \end{split}
 \end{equation*}
Since $|J| = n/2 - n^{2/3}$, this shows that
  \begin{equation*} 
   \frac{2 + o(1)}{n} \geq \sum _{i} \frac {|{\mathcal B}^{(i+1)} |}{\binom {n}{i+1}}
  \end{equation*}
giving $|{\mathcal B}| \leq \frac{2 + o(1)}{n} \binom {n}{n/2}$, as required. 
\end{proof}

\section{Proof of Theorem \ref{stickout k}}

The proof of Theorem \ref{stickout k} will use of the following result of Frankl and F\"uredi \cite{FrFu}.

\begin{thm}[Frankl-F\"uredi]
 \label{FF}
 Let $r,k\in \mathbb {N}$ with $0\leq k < r$. Suppose that ${\mathcal A} \subset [n]^{(r)}$ with $|A \cap B| \neq k$ for all $A,B \in {\mathcal A}$. Then $|{\mathcal A}| \leq d_r n^{\max {(k,r-k-1)}}$ where $d_r $ is a constant depending only on $r$.
\end{thm}

We will also make use of the Erd\H{o}s-Ko-Rado theorem \cite{EKR}.

\begin{thm}[Erd\H{o}s-Ko-Rado]
 \label{EKR}
 Suppose that $k\in \mathbb {N}$ and that $2k\leq n$. Then any family ${\mathcal A} \subset [n]^{(k)}$ with $A \cap B \neq \emptyset $ for all $A,B \in {\mathcal A}$ satisfies $|{\mathcal A}| \leq \binom {n-1}{k-1}$.
\end{thm}

We are now ready for the proof of the main result. Given a set $U \subset [n]$ and a permutation $\sigma \in S_n$, below we write $\sigma (U) = \{\sigma (u): u\in U\}$.

\begin{proof}[Proof of Theorem \ref{stickout k}.]
We will assume for convenience that $n$ is a multiple of $3k$ -- this assumption can easily be removed. To begin, remove all elements in ${\mathcal A}$ of size smaller than $n/2 - n^{2/3}$ or larger than 
 $n/2 + n^{2/3}$. By Chernoff's inequality (see Appendix A of \cite{alonspen}), we have removed at most $o(\frac{1}{n^k}\binom {n}{n/2})$ sets. Let ${\mathcal B}$ denote the remaining sets in $\mathcal A$. For each $l\in [0,k-1]$, let 
 \begin{equation*} 
 {\mathcal B}_l = \{B \in {\mathcal B}:|B| \equiv l\pmod k\}.
 \end{equation*} 
To prove the theorem it suffices to prove that for all $l\in [0,k-1]$ we have $|{\mathcal B}_l| \leq \frac {c'}{n^k}\binom {n}{n/2}$, where $c'=c'(k)>0$. We will show this when $l=0$ as the other cases are similar.

Let $I = [1, n/3]$ and $J = [n/3 +1,n]$ so 
that $[n] = I \cup J$. Let us choose a permutation $\sigma \in S_n$ uniformly at random. Given this choice of $\sigma $, for all $i\in [n/3k]$ and $S\in {J}^{(n/3)}$ let 
\begin{equation*}
C_{i,S} = \sigma (\{1,\ldots ,ik\}) \cup \sigma (S).
\end{equation*} 
Let ${\mathcal C}_S = \{C_{i,S}: i\in [n/3k]\}$ and call these sets `partial chains'. We write 
\begin{equation*}
{\mathcal D}  =\{ S \in \binom {J}{n/3}: \C _S \cap \B _0 \neq \emptyset \} \subset \binom {J}{n/3}.
\end{equation*} 

We claim that for any choice of $\sigma \in S_n$, we have
\begin{equation}
 \label{setdiffkineq}
 |{\mathcal D} | \leq \frac{d_{2k} (12k^2)^k}{n^k} \binom {|J|}{n/3},
\end{equation}
where $d_{2k}$ is as in Theorem \ref{FF}. Indeed otherwise, by averaging, there exists $T \in J^{(n/3 -2k)}$ for which the family 
\begin{equation*}
{\mathcal D} _T = \Big \{U \in (J\setminus T)^{(2k)}: U \cup T \in {\mathcal D} \Big \} \subset {(J\setminus T)}^{(2k)}
\end{equation*} 
satisfies $|{\mathcal D} _T| >  \frac{d_{2k} (12k^2)^k}{n^k} \binom {|J\setminus T|}{2k}$. This gives that
\begin{equation*}
 |{\mathcal D} _T| >  \frac{d_{2k} (12k^2)^k}{n^k} \binom {|J\setminus T|}{2k}  \geq \frac{d_{2k} (12k^2)^k}{n^k} \frac{|J\setminus T|^{2k}}{(2k)^{2k}} = \frac{d_{2k} |J\setminus T|^{2k}}{(n/3)^k} \geq {d_{2k}} |J\setminus T|^k,
\end{equation*}
since $|J\setminus T| = n/3 + 2k \geq n/3$. However, applying Theorem \ref{FF} to ${\mathcal D}_T$ with $r=2k$ we find $U, U'\in {\mathcal D_T}$ with $|U\cap U'| = k$. This then gives $C_{i,U\cup T}, C_{i',U'\cup T} \in {\mathcal B}_0$ for some $i,i'\in [n]$. Without loss of generality, we have $i\leq i'$. But then, as $\sigma (\{1,\ldots ,ik\}) \subset \sigma(\{1,\ldots ,i'k\})$, we have
\begin{equation*}
|C_{i,U\cup T} \setminus C_{i',U'\cup T}| = |\sigma (U) \setminus \sigma (U')| = |U\setminus U'| =|U| - |U \cap U'| = 2k - k = k.
\end{equation*} 
However $|A\setminus B| \neq k$ for all $A,B \in {\mathcal B}_0$. This contradiction shows that (\ref{setdiffkineq}) must hold.

Now the bound (\ref{setdiffkineq}) shows that for any choice of $\sigma \in S_n$, at most $c_k/n^k$ proportion of the sets ${\mathcal C}_S$ can contain elements of ${\mathcal B}_0$. Note however that any of these partial chains may still contain many elements from ${\mathcal B}_0$. As in the proof of Theorem \ref{stickoutexactly1}, we now show that this cannot happen too often.

For $i\in [n/3k]$ and $S\in {J}^{(n/3)}$, let $X_{i,S}$ denote the random variable given by
  \[X_{i,S} = \bigg\{ \begin{array}{ll}
         1 & \mbox{ if } C_{i,S} \in {\mathcal B}_0 \mbox{ and } C_{i',S} \notin {\mathcal B}_0 
               \mbox{ for all } i' < i;\\
         0 & \mbox{ otherwise.}\end{array} \]
 From the previous paragraph, we have 
  \begin{equation}
   \label{XiSinequality}
   \sum _{i,S} X_{i,S} \leq  \frac{d_{2k} (12k^2)^k}{n^k} \binom {|J|}{n/3}
  \end{equation}
 where both here and below the sum is taken over all $i \in [n/3k]$ and $S\in {J}^{(n/3)}$. Taking expectations on both sides of (\ref{Xijinequality}) this gives 
  \begin{equation}
   \label{expectationupperboundXiS}
   \sum _{i, S} \mathbb{E}(X_{i,S}) \leq  \frac{d_{2k} (12k^2)^k}{n^k} \binom {|J|}{n/3}.
  \end{equation}
 Rearranging we have
  \begin{equation}
   \label{expectationXiS}
   \begin{split}
   \sum _{i,S} \mathbb{E}(X_{i,S}) &= \sum _{i,S} \sum _{B \in {\mathcal B}_0} \mathbb {P}(C_{i,S} = B \mbox{ and } C_{i',S} \notin {\mathcal B}_0 
   \mbox{ for } i' <i).
   \end{split} 
 \end{equation}

 We now bound $\mathbb {P}(C_{i,S} = B \mbox{ and } C_{i',S}\notin {\mathcal B}_0 \mbox{ for } i'<i)$ for sets $B\in {\mathcal B}_0$. Note that we can only have $C_{i,S} = B$ if $|B| = ik+n/3$. Furthermore, for such $B$, since $C_{i,S}$ is equally likely to be any subset of $[n]$ of size $ik+n/3$, we have $\mathbb{P}(C_{i,S} = B) = 1 / \binom {n}{ik+n/3}$. We will prove that for all such $B$
 \begin{equation}
  \label{looselittlebyrestricting'}
  \mathbb {P}( C_{i,S} = B \mbox { and } C_{i',S} \notin {\mathcal B}_0 \mbox{ for } i'<i) 
  =  (1 - o(1)) \mathbb {P}( C_{i,S} = B )
 \end{equation}
To see this, note that given any set $D \subset [n]$ and two sets $E_1,E_2\in D^{(k)}$ for which $D\setminus E_1 ,D\setminus E_2\in {\mathcal B}_0$, we must have $E_1\cap E_2 \neq 0$ -- otherwise $|(D\setminus E_1) \setminus (D\setminus E_2)| = k$. Therefore, for $|D|\geq 2k$, by Theorem \ref{EKR}, there are at most $\binom {|D|-1}{k-1} = \frac{k}{|D|}\binom {|D|}{k}$ choices of $E \in {D}^{(k)}$ with $D \setminus E \in {\mathcal B}_0$. Recalling that  $C_{i',S} = C_{i,S} -\{\sigma (i'k+1),\ldots ,\sigma (ik)\}$ for all $i'<i$ and that $\{\sigma (i'k+1),\ldots ,\sigma ((i'+1)k)\}$ is chosen uniformly at random among all $k$-sets in $\{\sigma (1),\ldots ,\sigma ((i'+1)k)\}$, we see that for $(i' + 1)k + n/3 \geq (n/2 - n^{2/3})$ we have 
  \begin{equation}
   \label{CiSprobability}
   {\mathbb P}(C_{i',S} \notin {\mathcal B}_0 | C_{i'+1,S},\ldots , C_{i,S}) \geq (1 - \frac {k}{(i'+1)k}) \geq (1-\frac{k}{n/6 - n^{2/3}}).
  \end{equation}
 Also, since $\mathcal B _0$ contains no sets of size less than $n/2 - n^{2/3}$, for $(i' + 1)k + n/3 < (n/2 - n^{2/3})$ we have
  \begin{equation}
   \label{trivialboundCiS}
    {\mathbb P}(C_{i',S} \notin {\mathcal B}_0 | C_{i'+1,S},\ldots , C_{i,S}) = 1.
  \end{equation}
 But now by repeatedly applying (\ref{CiSprobability}) and (\ref{trivialboundCiS}), we get that 
 for any $B$ of size $ik + n/3\in [n/2 - n^{2/3}, n/2 + n^{2/3}]$ we have 
  \begin{equation*}
   \begin{split}
    {\mathbb P} ( C_{i,S} = B \mbox{ and } C_{i',S} \notin { \mathcal B }_0 \mbox{ for } i' < i ) 
     & \geq (1- \frac {k}{n/6 - n^{2/3}} )^{2n^{2/3}/k}{\mathbb P}( C_{i,S} = B )\\
     & \geq (1- \frac{k}{n/6 -n^{2/3}})^{2n^{2/3}/k} {\mathbb P}(C_{i,S} = B)\\
     & = (1 - o(1)){\mathbb P} ( C_{i,S} = B ).
   \end{split}
  \end{equation*}
 Now combining (\ref{looselittlebyrestricting'}) with (\ref{expectationupperboundXiS}) and (\ref{expectationXiS}) we obtain 
 \begin{equation*}
  \begin{split}
   \frac{d_{2k} (12k^2)^k}{n^k} \binom {|J|}{n/3} & \geq \sum _{i, S} \mathbb{E}(X_{i,S}) \\
 & = \sum _{i,S} \sum _{B \in {\mathcal B}_0} \mathbb {P}(C_{i,S} = B \mbox{ and } C_{i',S} \notin {\mathcal B}_0
   \mbox{ for } i' <i)\\
 & = \sum _{i,S} 
     \sum _{B \in {\mathcal B}_0^{(ik+n/3)}} (1 - o(1)) \mathbb {P}( C_{i,S} = B)\\
 & = (1 - o(1)) \sum _{i,S} \frac { |{\mathcal B}_0^{(ik+n/3)}| }{\binom {n}{ik+n/3}}\\
 & = (1-o(1))\binom {|J|}{n/3} \sum _{j\in [n]} \frac { |{\mathcal B}_0^{(j)}| }{\binom {n}{j}}.
  \end{split}
 \end{equation*}
But this shows that
  \begin{equation*} 
   \frac{d_{2k} (12k^2)^k}{n^k}  \geq \sum _{j\in [n]} \frac {|{\mathcal B}_0^{(j)} |}{\binom {n}{j}}
  \end{equation*}
giving $|{\mathcal B}_0| \leq  \frac{d_{2k} (12k^2)^k}{n^k} \binom {n}{n/2}$, as required. 
\end{proof}

\section{Concluding remarks}

It would be very interesting to determine the true answer in Theorem \ref{stickoutexactly1}, i.e. to remove the factor of 2. This is related to the well-known problem of finding the maximum size of a set system in which no two members are at
Hamming distance 2, where there is also a `gap' of multiplicative constant 2. Indeed, our proof of Theorem \ref{stickoutexactly1} can be modified to show that the answers to these two questions are asymptotically equal. See Katona \cite{katonagap} for background on this problem.

\section*{Acknowledgement}

We would like to thank the anonymous referee for suggesting pursuing the extension of Theorem \ref{stickoutexactly1} to Theorem \ref{stickout k}.


\end{document}